\documentclass[twoside,leqno,10pt, A4]{amsart}
\usepackage{amsfonts}
\usepackage{amsmath}
\usepackage{amscd}
\usepackage{amssymb}
\usepackage{amsthm}
\usepackage{amsrefs}
\usepackage{latexsym}
\usepackage{mathrsfs}
\usepackage{bbm}
\usepackage{amscd}
\usepackage{amssymb}
\usepackage{amsthm}
\usepackage{amsrefs}
\usepackage{latexsym}
\usepackage{mathrsfs}
\usepackage{bbm}
\usepackage{enumerate}
\usepackage{graphicx}
\usepackage{color}
\begin{document}

\newtheorem{theorem}[subsection]{Theorem}
\newtheorem{proposition}[subsection]{Proposition}
\newtheorem{lemma}[subsection]{Lemma}
\newtheorem{corollary}[subsection]{Corollary}
\newtheorem{conjecture}[subsection]{Conjecture}
\newtheorem{prop}[subsection]{Proposition}
\newtheorem{defin}[subsection]{Definition}

\numberwithin{equation}{section}
\newcommand{\mr}{\ensuremath{\mathbb R}}
\newcommand{\mc}{\ensuremath{\mathbb C}}
\newcommand{\dif}{\mathrm{d}}
\newcommand{\intz}{\mathbb{Z}}
\newcommand{\ratq}{\mathbb{Q}}
\newcommand{\natn}{\mathbb{N}}
\newcommand{\comc}{\mathbb{C}}
\newcommand{\rear}{\mathbb{R}}
\newcommand{\prip}{\mathbb{P}}
\newcommand{\uph}{\mathbb{H}}
\newcommand{\fief}{\mathbb{F}}
\newcommand{\majorarc}{\mathfrak{M}}
\newcommand{\minorarc}{\mathfrak{m}}
\newcommand{\sings}{\mathfrak{S}}
\newcommand{\fA}{\ensuremath{\mathfrak A}}
\newcommand{\mn}{\ensuremath{\mathbb N}}
\newcommand{\mq}{\ensuremath{\mathbb Q}}
\newcommand{\half}{\tfrac{1}{2}}
\newcommand{\f}{f\times \chi}
\newcommand{\summ}{\mathop{{\sum}^{\star}}}
\newcommand{\chiq}{\chi \bmod q}
\newcommand{\chidb}{\chi \bmod db}
\newcommand{\chid}{\chi \bmod d}
\newcommand{\sym}{\text{sym}^2}
\newcommand{\hhalf}{\tfrac{1}{2}}
\newcommand{\sumstar}{\sideset{}{^*}\sum}
\newcommand{\sumprime}{\sideset{}{'}\sum}
\newcommand{\sumprimeprime}{\sideset{}{''}\sum}
\newcommand{\sumflat}{\sideset{}{^\flat}\sum}
\newcommand{\shortmod}{\ensuremath{\negthickspace \negthickspace \negthickspace \pmod}}
\newcommand{\V}{V\left(\frac{nm}{q^2}\right)}
\newcommand{\sumi}{\mathop{{\sum}^{\dagger}}}
\newcommand{\mz}{\ensuremath{\mathbb Z}}
\newcommand{\leg}[2]{\left(\frac{#1}{#2}\right)}
\newcommand{\muK}{\mu_{\omega}}
\newcommand{\thalf}{\tfrac12}
\newcommand{\lp}{\left(}
\newcommand{\rp}{\right)}
\newcommand{\Lam}{\Lambda_{[i]}}
\newcommand{\lam}{\lambda}
\newcommand{\af}{\mathfrak{a}}
\newcommand{\sw}{S_{[i]}(X,Y;\Phi,\Psi)}
\newcommand{\lz}{\left(}
\newcommand{\pz}{\right)}
\newcommand{\bfrac}[2]{\lz\frac{#1}{#2}\pz}
\newcommand{\odd}{\mathrm{\ primary}}
\newcommand{\even}{\text{ even}}
\newcommand{\res}{\mathrm{Res}}
\newcommand{\sumn}{\sumstar_{(c,1+i)=1}  w\left( \frac {N(c)}X \right)}
\newcommand{\lab}{\left|}
\newcommand{\rab}{\right|}
\newcommand{\Go}{\Gamma_{o}}
\newcommand{\Ge}{\Gamma_{e}}
\newcommand{\M}{\widehat}

\theoremstyle{plain}
\newtheorem{conj}{Conjecture}
\newtheorem{remark}[subsection]{Remark}

\makeatletter
\def\widebreve{\mathpalette\wide@breve}
\def\wide@breve#1#2{\sbox\z@{$#1#2$}%
     \mathop{\vbox{\m@th\ialign{##\crcr
\kern0.08em\brevefill#1{0.8\wd\z@}\crcr\noalign{\nointerlineskip}%
                    $\hss#1#2\hss$\crcr}}}\limits}
\def\brevefill#1#2{$\m@th\sbox\tw@{$#1($}%
  \hss\resizebox{#2}{\wd\tw@}{\rotatebox[origin=c]{90}{\upshape(}}\hss$}
\makeatletter

\title[Ratios conjecture of cubic $L$-functions of prime moduli]{Ratios conjecture of cubic $L$-functions of prime moduli}

%%\date{\today}
\author[P. Gao]{Peng Gao}
\address{School of Mathematical Sciences, Beihang University, Beijing 100191, China}
\email{penggao@buaa.edu.cn}

\author[L. Zhao]{Liangyi Zhao}
\address{School of Mathematics and Statistics, University of New South Wales, Sydney NSW 2052, Australia}
\email{l.zhao@unsw.edu.au}

\begin{abstract}
We develop $L$-functions ratios conjecture with one shift in the numerator and denominator in certain ranges for the family of cubic Hecke $L$-functions of prime moduli over the Eisenstein field using multiple Dirichlet series under the generalized Riemann hypothesis. As applications, we evaluate asymptotically the first moment of central values as well as the one-level density of the same family of $L$-functions.
\end{abstract}

\maketitle



\section{Introduction}\label{sec 1}

  Originated in the work of D. W. Farmer \cite{Farmer93} on the shifted moments of the Riemann zeta function and formulated for general $L$-functions by J. B. Conrey, D. W. Farmer and M. R. Zirnbauer in \cite[Section 5]{CFZ}, the $L$-functions ratios conjecture makes predictions on the asymptotic behaviors of the sum of ratios of products of shifted $L$-functions. This conjecture has many important applications including the studies of moments of $L$-functions and the density conjecture of N. Katz and P. Sarnak \cites{KS1, K&S} on the distribution of zeros near the central point of a family of $L$-functions. \newline

  There are now a few results in the literature developing the ratios conjecture, starting from the work of H. M. Bui, A. Florea and J. P. Keating \cite{BFK21} on quadratic $L$-functions over function fields.  A somewhat more general result on the ratios conjecture for quadratic $L$-functions over function fields was recently obtained by V. Y. Wang \cite{VYWang} using different methods.  The first available result on ratios conjecture over number fields is given by M. \v Cech \cite{Cech1}, who considered the families of both general and primitive quadratic Dirichlet $L$-functions under the assumption of the generalized Riemann hypothesis (GRH). \newline

  A powerful tool used in the work of \v Cech \cite{Cech1} is the method of multiple Dirichlet series, an approach that has been
previous employed mainly to study moments of central values of families of $L$-functions. The same approach was then applied by the authors \cites{G&Zhao14, G&Zhao15, G&Zhao16} to develop ratios
conjecture for various families of $L$-functions.  \newline

  So far, all known results on ratios conjecture concentrate on families of $L$-functions whose conductors span a subset of positive density over the ring of integers of the corresponding number field. Given that families of $L$-functions whose conductors form a sparse set (a set of density zero) in the ring of integers are also important to study (for example, families of $L$-functions of prime conductors), it is therefore also interesting to develop ratios conjecture for these families. The aim of this paper is to do so for one such case. More specifically, we examine the ratios conjecture for cubic Hecke $L$-functions over the Eisenstein field $\mq(\omega)$, where $\omega=\frac {-1+\sqrt{3}i}{2}$ is a primitive cubic root of unity. \newline

  To state the result, we henceforth fix $K=\mq(\omega)$ for the Eisenstein field.  We write $\mathcal{O}_K$ for the ring of integers and  $U_K$ for the group of units in $\mathcal{O}_K$. We shall reserve the symbol $\varpi$ for a prime number in $\mathcal{O}_K$, by which we mean that the ideal $(\varpi)$ generated by $\varpi$ is a prime ideal.  Let $N(n)$ stand for the norm of any $n \in K$ and $\chi$ for a Hecke character of $K$ and we say that $\chi$ is of trivial infinite type if its component at the infinite places of $K$ is trivial. We write $L(s,\chi)$ for the $L$-function associated to $\chi$ and $\zeta_{K}(s)$ for the Dedekind zeta function of $K$ and $\zeta(s)$ for the Riemann zeta function.   We also use the notation $L^{(c)}(s, \chi_m)$ for the Euler product defining $L(s, \chi_m)$ but omitting those primes dividing $c$.  Let $\Lambda_K(n)$ be the von Mangoldt function on $\mathcal{O}_K$ given by
\begin{align*}
    \Lambda_K(n)=\begin{cases}
   \log N(\varpi) \qquad & n=\varpi^k, \text{$\varpi$ prime}, k \geq 1, \\
     0 \qquad & \text{otherwise}.
    \end{cases}
\end{align*}

  We fix $j=3$ throughout the paper and we note that (see \cite[Proposition 8.1.4]{BEW}) every ideal in $\mathcal{O}_K$ co-prime to $j$ has a unique generator congruent to $1$ modulo $3$ which is called primary. Let $\chi_{j, \varpi}$ be a cubic Hecke symbol defined in Section \ref{sec2.4}.  It is shown there that $\chi_{j, \varpi}$ is a Hecke character of trivial infinite type when $\varpi \equiv 1 \pmod {9}$.  We shall detect this congrudence using ray class group
characters. Here we call that for any integral ideal $\mathfrak{m} \in O_K$, the ray class group $h_{\mathfrak{m} }$ is defined to be $I_{\mathfrak{m} }/P_{\mathfrak{m} }$, where $I_{\mathfrak{m} } = \{
\mathcal{A} \in I : (\mathcal{A}, \mathfrak{m} ) = 1 \}$ and $P_{\mathfrak{m} } = \{(a) \in P : a \equiv 1 \bmod \mathfrak{m}  \}$ with $I$ and $P$ denoting the group of
fractional ideals in $K$ and the subgroup of principal ideals, respectively. \newline

   Our main result investigates, under GRH, the ratios conjecture with one shift in the numerator
and denominator for the family of cubic Hecke $L$-functions of prime moduli. As usual, in the rest of the paper, we use $\varepsilon$ to denote a small positive quantity which may not be the same at each occurrence.
\begin{theorem}
\label{Theorem for all characters}
With the notation as above and assuming the truth of GRH, let $X$ be a large real number and $w(t)$ be a non-negative Schwartz function and $\widehat w(s)$ be its Mellin transform.  Set
\begin{align}
\label{Nab}
			E(\alpha,\beta)=\max\left\{\tfrac 12, \ \tfrac 56-\Re(\alpha), \ 1-\Re(\beta),  \ 1-3\Re(\alpha)-2\Re(\beta), \ 1-\tfrac {13}{15}\Re(\alpha)-\tfrac {2}{15}\Re(\beta), \ \tfrac {12}{13}-\tfrac {11}{13}\Re(\alpha) \right\}.
\end{align}
and $\delta(\alpha)=2/11$ for $\Re(\alpha) <0$ and $\delta(\alpha)=0$ otherwise.  Then for $\Re(\alpha) > -1/11$ and $\Re(\beta)>0$ such that $E(\alpha,\beta)<1$,
\begin{align}
\label{Asymptotic for ratios of all characters}
\begin{split}	
\sumstar_{\substack{\varpi}} & \frac {\Lambda_{K}(\varpi) L(\tfrac 12+\alpha, \chi_{j, \varpi})}{L(\tfrac 12+\beta, \chi_{j, \varpi})} w \bfrac {N(\varpi)}X  \\
& \hspace*{2cm} = \frac {\M w(1)X}{\# h_{(9)}}  \frac{1-3^{-1/2-\beta}}{1-3^{-1/2-\alpha}}
\frac {\zeta^{(j)}_K( \tfrac{3}{2}+3\alpha)}{\zeta^{(j)}_K( \tfrac{3}{2}+2\alpha+\beta)} +O\lz(1+|\alpha|)^{\delta(\alpha)+\varepsilon}(1+|\beta|)^{\varepsilon} X^{E(\alpha,\beta)+\varepsilon}\pz,
\end{split}
\end{align}
 where the ``$*$'' on the sum over $\varpi$ means that the sum is restricted to prime elements $\varpi$ of $\mathcal{O}_K $ that are congruent to $1$ modulo $9$.
\end{theorem}

   Our result above agrees with the prediction from the ratios conjecture on the left-hand side of \eqref{Asymptotic for ratios of all characters}, which can be derived following the recipe given in \cite[Section 5]{CFZ} except that (see also \cite{CS}) the ratios conjecture asserts that \eqref{Asymptotic for ratios of all characters} holds uniformly for $|\Re(\alpha)|< 1/4$, $(\log X)^{-1} \ll \Re(\beta) < 1/4$ and $\Im(\alpha), \Im(\beta) \ll X^{1-\varepsilon}$ with an error term $O(X^{1/2+\varepsilon})$. On the other hand, Theorem \ref{Theorem for all characters} has the advantage that there is no constraint on imaginary parts of $\alpha$ and $\beta$.   In fact, under the conditions of Theorem~\ref{Theorem for all characters}, \eqref{Asymptotic for ratios of all characters} gives an asymptotic formula as soon as $\Re (\beta) > -3\Re (\alpha)/2 + \varepsilon$ holds. \newline

   Upon taking the limit $\beta \rightarrow \infty$ on both sides of \eqref{Asymptotic for ratios of all characters} and observing that in this case one may drop the factor involving with $\beta$ in the error term, we readily obtain the following result concerning the first moment of cubic Hecke $L$-functions.
\begin{theorem}
\label{Thmfirstmoment}
		With the notation as above and assuming the truth of GRH, we have for $\Re(\alpha)>-1/11$,
\begin{align}
\label{Asymptotic for first moment}
\begin{split}			
& \sumstar_{\substack{\varpi}}\Lambda_{K}(\varpi)L(\tfrac 12+\alpha, \chi_{j, \varpi}) w \bfrac {N(\varpi)}X  =  \frac {\M w(1)X}{\# h_{(9)}}
 \frac{\zeta^{(j)}_K(\tfrac{3}{2}+3\alpha)}{1-3^{-1/2-\alpha}} +O\lz(1+|\alpha|)^{\delta(\alpha)+\varepsilon}X^{E(\alpha)+\varepsilon}\pz,
\end{split}
\end{align}
  where
\begin{align}
\label{Ealpha}
\begin{split}
      &  E(\alpha):=\lim_{\beta \rightarrow \infty}E(\alpha,\beta)=\max\left\{\tfrac 12, \ \tfrac 56-\Re(\alpha), \ \tfrac {12}{13}-\tfrac {11}{13}\Re(\alpha)  \right\}.
\end{split}
\end{align} 	
\end{theorem}

 We remark here that the case $\alpha=0$ in \eqref{Asymptotic for first moment} establishes an asymptotic formula for the first moment of central values of the family of cubic Hecke $L$-functions of prime  moduli with an error term of size $O(X^{12/13+\varepsilon})$, which improves the cubic case given in \cite[Theorem 1.2]{G&Zhao2022-4}. \newline

  Similarly, we set $\alpha \rightarrow \infty$ on both sides of \eqref{Asymptotic for ratios of all characters} and drop the factor involving with $\alpha$ in the error term to obtain the following result concerning the negative first moment of cubic Hecke $L$-functions.
\begin{theorem}
\label{Thmnegfirstmoment}
		With the notation as above and assuming the truth of GRH, we have for $\Re(\beta)>0$,
\begin{align*}
%%\label{Asymptotic for neg first moment}
\begin{split}			
& \sumstar_{\substack{\varpi}}\frac {\Lambda_{K}(\varpi)}{L(\tfrac 12+\beta, \chi_{j, \varpi})} w \bfrac {N(\varpi)}X  =  \frac {\M w(1)X}{\# h_{(9)}}(1-3^{-1/2-\beta})
 +O\lz (1+|\beta|)^{\varepsilon} X^{\max\left\{1/2, 1-\Re(\beta) \right\}+\varepsilon}\pz.
\end{split}
\end{align*}
\end{theorem}

Of course, Theorem~\ref{Thmnegfirstmoment} offers an asymptotic formula whenever $\Re (\beta) > \varepsilon$.  In section \ref{sec: logder}, we shall also differentiate with respect to $\alpha$ in \eqref{Asymptotic for ratios of all characters} and then set $\alpha=\beta=r$ to obtain an asymptotic formula for the smoothed first moment of $L'(\frac{1}{2}+r,\chi_{j, \varpi})/L(\frac{1}{2}+r,\chi_{j, \varpi})$ averaged over $\varpi \equiv 1 \pmod 9$.
\begin{theorem}
\label{Theorem for log derivatives}
	With the notation as above and assuming the truth of GRH, we have for $0<\varepsilon< \Re(r)<1/2$,
\begin{align}
\label{Sum of L'/L with removed 2-factors}
\begin{split}
		& \sumstar_{\substack{\varpi}}\frac {\Lambda_{K}(\varpi)L'(\tfrac 12+r, \chi_{j, \varpi})}{L(\tfrac 12+r, \chi_{j, \varpi})}w \bfrac {N(\varpi)}X
			= \frac {\M w(1)X }{\# h_{(9)}}\lz\frac{(\zeta^{(j)}_K(\frac 32+3r))'}{\zeta^{(j)}_K(\frac 32+3r)}-\frac{\log 3 }{3^{1/2+r}-1}\pz+O((1+|r|)^{\varepsilon}X^{1-\Re(r)+\varepsilon}).
\end{split}
\end{align}
\end{theorem}
	
 Now, Theorem \ref{Theorem for log derivatives} allows one to compute the one-level density of low-lying zeros of the corresponding families of cubic Hecke $L$-functions, following the approach in the proof of \cite[Corollary 1.5]{Cech1} using \cite[Theorem 1.4]{Cech1} for the case of quadratic Dirichlet $L$-functions. For this, let $h(x)$ be an even Schwartz function such that its Fourier transform is supported in the interval $[-a,a]$ for some $a>0$.
We define the one-level density of the family of $L$-functions of interest in this paper by
\begin{align*}
%%\label{Ddef}
\begin{split}
		D(X;h)=\frac{1}{F(X)}\sumstar_{\substack{ \varpi \odd}}\Lambda_K(\varpi)w \bfrac {N(\varpi)}X\sum_{\gamma_{\varpi, n}}h\bfrac{\gamma_{\varpi, n}\log X}{2\pi},
\end{split}
\end{align*}
where $\gamma_{\varpi, n}$ runs over the imaginary parts of the non-trivial zeros of $L(s,\chi_{j, \varpi})$ and
\begin{align}
\label{Size of family}
\begin{split}
			F(X)&=\sumstar_{\substack{\varpi \odd}}\Lambda_K(\varpi) w \bfrac {N(\varpi)}X.
\end{split}
\end{align}
	
   With the aid of Theorem \ref{Theorem for log derivatives}, our next result computes $D(X;h)$ asymptotically.
\begin{theorem}
\label{Theorem one-level density}
	With the notation as above and assuming the truth of GRH, for any function $w(t)$ that is non-negative and compactly supported on the set of positive real numbers, we have
\begin{align}
\label{Onelevel}
\begin{split}
		D(X;h)=& \frac {2\widehat h(1)}{F(X)\log X} \sumstar_{\substack{\varpi \odd}}\Lambda_K(\varpi) w \bfrac {N(\varpi)}X\log N(\varpi) \\
& \hspace*{0.5cm} +\frac{1}{F(X)}\frac {2 \M w(1) }{\# h_{(9)}} \cdot \frac{X}{\log X}\int\limits_{-\infty}^{\infty}h(u)\lz\frac{(\zeta^{(j)}_K( \tfrac{3}{2} + \frac {6\pi iu}{\log X}))'}{\zeta^{(j)}_K( \tfrac{3}{2} +\frac {6\pi iu}{\log X})}-\frac{\log 3 }{3^{1/2+2\pi iu/\log X}-1}\pz \dif u \\
& \hspace*{0.5cm} +\frac 1{\log X}\int\limits^{\infty}_{-\infty}h\lz u \pz \lz \frac {\Gamma'}{\Gamma} \left( \frac 12+\frac {2 \pi iu}{\log X} \right) +\frac {\Gamma'}{\Gamma} \left( \frac 12-\frac {2 \pi iu}{\log X} \right) \pz \dif u+\frac {2\widehat h(1)}{\log X} \log \frac {|D_K|}{4\pi^2} +O(X^{(1+a)/2+\varepsilon}).
\end{split}
\end{align}
  In particular, if $a<1$, we have
\begin{align}
\label{Onelevelasym}
\begin{split}
		D(X;h)= \int\limits^{\infty}_{-\infty}h(x) \dif x & +\frac{2\widehat h(1)}{\log X} \lz \frac{2(\zeta^{(j)}_K(\tfrac{3}{2}))'}{\zeta^{(j)}_K(\tfrac{3}{2})}-\frac{2 \log 3 }{3^{\frac 12}-1} + \frac {1}{\widehat w(1)} \left( \int\limits^{\infty}_0 w(u)\log u \ \dif u \right)+2\frac {\Gamma'}{\Gamma}\left( \frac 12 \right) +
\log \frac {|D_K|}{4\pi^2} \pz   \\
		& +O \left( \frac 1{(\log X)^2} \right).
\end{split}
\end{align}
\end{theorem}

   We refer the readers to the work of C. David  and A. M. G\"{u}lo\u{g}lu \cites{DG22} on one-level density results concerning families of cubic Hecke $L$-functions whose conductors form sets of positive density in $\mathcal O_K$. In fact, besides using the approach of multiple Dirichlet series, the proof of
Theorem \ref{Theorem for all characters} also relies heavily on \cite[Theorem 4.4]{DG22}, which estimates sums of cubic Gauss sums over
primes. \newline

  Our considerations above also extend to the case of cubic Dirichlet $L$-functions. To be more precise, we reserve the letter $p$ for a rational prime.  Let $\Lambda(n)$ denote the usual von Mangoldt function on $\mz$ and $\chi_0$ the principal Dirichlet character. With minor modifications, our results on cubic
Hecke $L$-functions carry over to the family of cubic Dirichlet $L$-functions whose conductors are rational primes that are norms of primes $\varpi \in \mathcal O_K$ such that $\varpi \equiv 1 \pmod 9$. We have the following result concerning the ratios conjecture of this family of $L$-functions.
\begin{theorem}
\label{Theorem for all charactersQ}
	With the notation as above and assuming the truth of GRH, let $Q$ be a large real number and $\Phi(t)$ be a non-negative Schwartz function with Mellin transform $\widehat \Phi(s)$.  Let $E(\alpha,\beta)$ be as in \eqref{Nab} and $\delta(\alpha)$ as in Theorem \ref{Theorem for all characters}.  Then we have for $\Re(\alpha) > -1/11$ and $\Re(\beta)>0$ such that $E(\alpha,\beta)<1$,
\begin{align}
\label{Asymptotic for ratios of all charactersQ}
\begin{split}	
\sum_{\substack{ p=N(\varpi) \\ \varpi \equiv 1 \bmod {9}} }\; &
\sum_{\substack{\chi \bmod{p} \\ \chi^3 = \chi_0, \ \chi \neq \chi_0}} \Lambda(p) \frac {L(\tfrac{1}{2}+\alpha, \chi)}{L(\tfrac{1}{2}+\beta, \chi)}  \Phi \leg{p}{Q}  \\
=&  \frac {2\M \Phi(1)Q}{\# h_{(9)}} \frac{1-3^{-1/2-\beta}}{1-3^{-1/2-\alpha}}
\frac {\zeta^{(j)}(\tfrac{3}{2}+3\alpha)}{\zeta^{(j)}(\tfrac{3}{2}+2\alpha+\beta)} +O\lz(1+|\alpha|)^{\delta(\alpha)/2+\varepsilon}(1+|\beta|)^{\varepsilon} Q^{E(\alpha,\beta)+\varepsilon}\pz.
\end{split}
\end{align}
\end{theorem}

  Again, taking $\beta \rightarrow \infty$ on both sides of \eqref{Asymptotic for ratios of all charactersQ} leads to the following result concerning the first moment of cubic Dirichlet $L$-functions.
\begin{theorem}
\label{ThmfirstmomentQ}
With the notation as above, assuming the truth of GRH and together with $E(\alpha)$ defined as in \eqref{Ealpha}, we have for $\Re(\alpha)>-1/11$,
\begin{align*}
%%\label{Asymptotic for first momentQ}
\begin{split}			
& \sum_{\substack{ p=N(\varpi) \\ \varpi \equiv 1 \bmod {9}} }\;
\sum_{\substack{\chi \bmod{p} \\ \chi^3 = \chi_0, \ \chi \neq \chi_0}} \Lambda(p) L(\tfrac 12+\alpha, \chi) \Phi \leg{p}{Q} =  \frac {2 \M \Phi(1)Q}{\# h_{(9)}}
 \frac{\zeta^{(j)}(\tfrac{3}{2}+3\alpha)}{1-3^{-1/2-\alpha}} +O\lz(1+|\alpha|)^{\delta(\alpha)/2+\varepsilon}Q^{E(\alpha)+\varepsilon}\pz.
\end{split}
\end{align*}	
\end{theorem}

   Similarly,  taking $\alpha \rightarrow \infty$ on both sides of \eqref{Asymptotic for ratios of all charactersQ} leads to the following result concerning the negative first moment of cubic Dirichlet $L$-functions.
\begin{theorem}
\label{ThmnegfirstmomentQ}
		With the notation as above and assuming the truth of GRH, we have for $\Re(\beta)>0$,
\begin{align*}
%%\label{Asymptotic for neg first momentQ}
\begin{split}			
& \sum_{\substack{ p=N(\varpi) \\ \varpi \equiv 1 \bmod {9}} }\;
\sum_{\substack{\chi \bmod{p} \\ \chi^3 = \chi_0, \ \chi \neq \chi_0}} \frac {\Lambda(p)}{L(\tfrac 12+\beta, \chi)} \Phi \bfrac {p}Q  =  \frac {2 \M \Phi(1)Q}{\# h_{(9)}}(1-3^{-1/2-\beta})
 +O\lz (1+|\beta|)^{\varepsilon} Q^{\max\left\{1/2, \ 1-\Re(\beta) \right\}+\varepsilon}\pz.
\end{split}
\end{align*}
\end{theorem}

   Also, we may differentiate with respect to $\alpha$ in \eqref{Asymptotic for ratios of all charactersQ} and set $\alpha=\beta=r$ to obtain an asymptotic formula for the smoothed first moment of the logarithmic derive of the $L$-functions under consideration.
 \begin{theorem}
\label{Theorem for log derivativesQ}
	With the notation as above and assuming the truth of GRH, we have for $0<\varepsilon< \Re(r)<1/2$,
\begin{align*}
%%\label{Sum of L'/L with removed 2-factorsQ}
\begin{split}
		& \sum_{\substack{ p=N(\varpi) \\ \varpi \equiv 1 \bmod {9}} }\;
\sum_{\substack{\chi \bmod{p} \\ \chi^3 = \chi_0, \ \chi \neq \chi_0}} \frac {\Lambda(p)L'(\tfrac 12+r, \chi)}{L(\tfrac 12+r, \chi)}\Phi \bfrac {p}Q
			= \frac {2 \M \Phi(1)Q }{\# h_{(9)}}\lz\frac{(\zeta^{(j)}(\frac 32+3r))'}{\zeta^{(j)}(\frac 32+3r)}-\frac{\log 3 }{3^{1/2+r}-1}\pz+O((1+|r|)^{\varepsilon}Q^{1-\Re(r)+\varepsilon}).
\end{split}
\end{align*}
\end{theorem}

   Further, Theorem \ref{Theorem for log derivativesQ} allows one to compute the one-level density of low-lying zeros of the corresponding families of cubic Dirichlet $L$-functions. Let $h(x)$ be an even Schwartz function taken earlier whose Fourier transform being supported in $[-a,a]$ for some $a>0$, then we define the one-level density of the family of cubic Dirichlet $L$-functions under consideration by
\begin{align*}
%%\label{DdefQ}
\begin{split}
		D_{\mq}(Q;h)=\frac{1}{F_{\mq}(Q)}\sum_{\substack{ p=N(\varpi) \\ \varpi \equiv 1 \bmod {9}} }\;
\sum_{\substack{\chi \bmod{p} \\ \chi^3 = \chi_0, \ \chi \neq \chi_0}} \Lambda(p) \Phi \bfrac {p}Q\sum_{\gamma_{\chi, n}}h\bfrac{\gamma_{\chi, n}\log Q}{2\pi},
\end{split}
\end{align*}
   where $\gamma_{\chi, n}$ runs over the imaginary parts of the non-trivial zeros of $L(s,\chi)$, and
\begin{align*}
%%\label{Size of familyQ}
\begin{split}
		F_{\mq}(Q)&=\sum_{\substack{ p=N(\varpi) \\ \varpi \equiv 1 \bmod {9}} }\;
\sum_{\substack{\chi \bmod{p} \\ \chi^3 = \chi_0, \ \chi \neq \chi_0}} \Lambda(p) \Phi \bfrac {p}Q.
\end{split}
\end{align*}
	
   Our conclude this section with the following result that computes $D_{\mq}(X;h)$ asymptotically.
\begin{theorem}
\label{Theorem one-level densityQ}
	With the notation as above and assuming the truth of GRH, for any function $\Phi(t)$ that is non-negative and compactly supported on the set of positive real numbers, we have
\begin{align*}
%%\label{OnelevelQ}
\begin{split}
		D_{\mq}(Q;h)=\frac {2\widehat h(1)}{F_{\mq}(Q)\log Q} & \sum_{\substack{ p=N(\varpi) \\ \varpi \equiv 1 \bmod {9}} }\;
\sum_{\substack{\chi \bmod{p} \\ \chi^3 = \chi_0, \ \chi \neq \chi_0}} \Lambda(p) \Phi \bfrac {p}Q \log p \\
& +\frac{1}{F_{\mq}(Q)}\frac {4 \M w(1) }{\# h_{(9)}} \cdot \frac{Q}{\log Q}\int\limits_{-\infty}^{\infty}h(u)\lz\frac{(\zeta^{(j)}(\tfrac{3}{2}+ \frac {6\pi iu}{\log Q}))'}{\zeta^{(j)}(\tfrac{3}{2}+\frac {6\pi iu}{\log Q})}-\frac{\log 3 }{3^{1/2+ 2\pi iu/\log Q}-1}\pz \dif u \\
&+\frac 1{\log Q}\int\limits^{\infty}_{-\infty}h\lz u \pz \lz \frac {\Gamma'}{\Gamma} \left( \frac 14+\frac {\pi iu}{\log Q} \right) +\frac {\Gamma'}{\Gamma} \left( \frac 14-\frac {\pi iu}{\log Q} \right) \pz \dif u+\frac {2\widehat h(1)}{\log Q} \log \frac {1}{\pi} +O(Q^{(1+a)/2+\varepsilon}).
\end{split}
\end{align*}
  In particular, for $a<1$, we have
\begin{align*}
%%\label{OnelevelasymQ}
\begin{split}
		D_{\mq}(Q;h)= \int\limits^{\infty}_{-\infty}h(x) \dif x & +\frac{2\widehat h(1)}{\log Q} \lz \frac{2(\zeta^{(j)}(\frac 32))'}{\zeta^{(j)}(\frac 32)}-\frac{2\log 3 }{3^{1/2}-1}  + \frac {1}{\widehat w(1)} \left( \int\limits^{\infty}_0w(u)\log u \dif u \right)+2\frac {\Gamma'}{\Gamma} \left( \frac 14 \right) +\log \frac {1}{\pi} \pz \\
		& +O \left( \frac 1{(\log Q)^2} \right).
\end{split}
\end{align*}
\end{theorem}

\section{Preliminaries}
\label{sec 2}

%%----------------------------------------------------------------------------
\subsection{Cubic residue symbols and Gauss sums}
\label{sec2.4}
%%----------------------------------------------------------------------------
   Recall that we write $K=\mq(\omega)$ throughout the paper. It is well-known that $K$ has class number one and $\mathcal{O}_{K}=\mz[\omega]$.
We use $D_K=-3$ for the different and discriminant of $K$, respectively.
The cubic residue symbol $\leg {\cdot}{\cdot}_3$  is defined for any prime $\varpi$ co-prime to $3$ in $\mathcal{O}_{K}$, such that
 we have $\leg{a}{\varpi}_3 \equiv
a^{(N(\varpi)-1)/3} \pmod{\varpi}$ with $\leg{a}{\varpi}_3 \in \{
1, \omega, \omega^2 \}$ for any $a \in \mathcal{O}_{K}$, $(a, \varpi)=1$. We also define
$\leg{a}{\varpi}_3 =0$ when $\varpi | a$. The definition for the cubic symbol is then extended
to any composite $n$ with $(N(n), 3)=1$ multiplicatively. We further define $\leg{a}{u}_3=1$ for any $a \in \mathcal O_K, u \in U_K$. Recall that we reserve
the letter $j$ for $3$ so that we shall henceforth write $\chi_{j, a}$ for the symbol $\leg {\cdot}{a}_{3}$ and $\chi^{(a)}_j$ for the symbol $\leg {a}{\cdot}_{3}$ for any $a \in \mathcal{O}_K$ such that $(N(a), 3)=1$. \newline

Recall that every ideal in $\mathcal O_K$ co-prime to $j$ has a unique primary generator congruent to $1$ modulo $3$. Moreover, an element $n=a+b\omega$ in $\mz[\omega]$ and $n \equiv 1 \pmod{3}$ if and only if $a \equiv 1 \pmod{3}$, and $b \equiv
0 \pmod{3}$ (see the discussions before \cite[Proposition 9.3.5]{I&R}). \newline

  We refer the reader to \cite[Proposition 3.7]{Lemmermeyer} for the cubic reciprocity law as well as the supplementary laws concerning the cubic residue symbol. We only point out here that the symbol $\chi_{j, a}$ is trivial on $U_K$ for any $a \equiv 1 \pmod {9}$. Moreover, notice that the ideal $(1-\omega)$ is the only ideal in $\mathcal O_K$ that lies above the rational ideal $(3)$ in $\mz$ and that the supplement laws to the cubic reciprocity law (see \cite[Proposition 3.7]{Lemmermeyer}) imply that
\begin{align}
\label{cubicsupplyment}
 \leg {1-\omega}{c}_{3}=1, \quad c \equiv 1 \pmod {9}.
\end{align}

  It follows that $\chi_{j, a}$  can be regarded as a primitive Hecke character modulo $\varpi$ of trivial infinite type. Note also that the symbol $\chi^{(a)}_{j}$  is also a Hecke character modulo $a$ of trivial infinite type. In the rest of the paper, we shall adapt the point of view to by treating both $\chi_{j, \varpi}$ and $\chi^{(a)}_{j}$  as Hecke characters.  \newline

  Let $e(z) = \exp (2 \pi i z)$ for any complex number $z$ and
\begin{align*}
   \widetilde{e}_K(z) =\exp \left( 2\pi i  \left( \frac {z}{\sqrt{D_K}} -
\frac {\overline{z}}{\sqrt{D_K}} \right) \right).
\end{align*}
  For any Hecke character $\chi$ modulo $q$ of trivial infinite type and any $k \in \mathcal O_K$, we define the associated Gauss sum $g_K(k, \chi)$ by
\begin{align*}
%%\label{g2}
 g_K(k,\chi) = \sum_{x \shortmod{q}} \chi(x) \widetilde{e}_K\leg{kx}{q}.
\end{align*}
  Note that our definition above for $g_K(k,\chi)$ is independent of the choice of a generator for $(q)$ since any Hecke character of trivial infinite type
is trivial on $U_K$ (see \cite[(3.80)]{iwakow}). Observe that we have
\begin{align}
\label{2.7}
   g_K(rs,\chi)=\overline{\chi}(s) g_K(r,\chi), \qquad (s,n)=1.
\end{align}

   We write $g_K(\chi)$ for $g_K(1,\chi)$ and $g_{K, j}(\varpi), g_{K,j}(k, \varpi)$ for $g_K(\chi_{j, \varpi}), g_K(k, \chi_{j, \varpi})$,
respectively.

\subsection{Functional equations for Hecke $L$-functions}
	
	For any primitive Hecke character $\chi$ of trivial infinite type modulo $q$, a well-known result of E. Hecke shows that $L(s, \chi)$ has an
analytic continuation to the whole complex plane and satisfies the
functional equation (see \cite[Theorem 3.8]{iwakow})
\begin{align}
\label{fneqn}
  \Lambda(s, \chi) = W(\chi)\Lambda(1-s, \overline \chi), \; \mbox{where} \;  W(\chi) = g_K(\chi)(N(q))^{-1/2}
\end{align}
and
\begin{align}
\label{Lambda}
  \Lambda(s, \chi) = (|D_K|N(q))^{s/2}(2\pi)^{-s}\Gamma(s)L(s, \chi).
\end{align}

  We apply \eqref{fneqn} and\eqref{Lambda} to the case $\chi=\chi_{j, \varpi}$ for a primary prime $\varpi \equiv 1 \pmod 9$ to deduce that
\begin{align}
\label{fneqnL}
  L(s, \chi_{j, \varpi})=g_{K,j}(\varpi)|D_K|^{1/2-s}N(\varpi)^{-s}(2\pi)^{2s-1}\frac {\Gamma(1-s)}{\Gamma (s)}L(1-s, \overline \chi_{j, \varpi}).
\end{align}

  Note that Stirling's formula (see \cite[(5.113)]{iwakow}) yields for constants $c_0, d_0$,
\begin{align}
\label{Stirlingratio}
  \frac {\Gamma(c_0(1-s)+ d_0)}{\Gamma (c_0s+ d_0)} \ll (1+|s|)^{c_0(1-2\Re (s))}.
\end{align}

\subsection{Estimations on $L$-functions}

  Recall that for any Hecke character $\chi$ of trivial infinite type, the corresponding $L$-function $L(s, \chi)$ has an Euler product for $\Re(s)$ large enough given by
\begin{align*}
%%\label{LEuler}
 L(s, \chi)=\prod_{(\varpi)} \left( 1-\frac {\chi(\varpi)}{N(\varpi)^s} \right)^{-1},
\end{align*}
  where we denote $(n)$ for the ideal generated by any $n \in \mathcal O_K$ throughout the paper.  Taking the logarithmic derivatives both sides above then gives that for $\Re(s)$ large enough,
\begin{align*}
%%\label{Llogder}
 -\frac {L'(s, \chi)}{L(s, \chi)}=\sum_{(n)}\frac {\Lambda_{K}(n)\chi(n)}{N(n)^s}.
\end{align*}

  Now let $\widehat \chi^{(m)}_j$ be the primitive character inducing $\chi^{(m)}_j$ so that
\begin{align*}
%%\label{Aswmain2}
\begin{split}
  L(s, \chi^{(m)}_j)=L(s, \widehat \chi^{(m)}_j)\prod_{\substack{ (\varpi) \\ \varpi|m}}(1-\widehat \chi^{(m)}_j(\varpi)N(\varpi)^{-s}).
\end{split}
\end{align*}
This leads to
\begin{align*}
%%\label{Llogderprim}
-\frac {L'(s, \chi^{(m)}_j)}{L(s, \chi^{(m)}_j)}=-\frac {L'(s, \widehat \chi^{(m)}_j)}{L(s, \widehat \chi^{(m)}_j)}-\sum_{\substack{ (\varpi) \\ \varpi|m}}(1-\widehat \chi^{(m)}_j(\varpi)N(\varpi)^{-s})^{-1}\widehat \chi^{(m)}_j(\varpi)(\log N(\varpi))N(\varpi)^{-s}.
\end{align*}

   It follows from  \cite[Theorem 5.17]{iwakow} that under GRH, for $\Re(s) \geq 1/2+\varepsilon$,
\begin{align} \label{Lderbound}
  -(s-1) \frac {L'(s, \widehat \chi^{(m)}_j)}{L(s, \widehat \chi^{(m)}_j)}  \ll 1+|s-1|\log \big ((N(m)+2)(1+|s|)\big).
\end{align}

Let $\mathcal{W}_K(n)$ denote the number of distinct prime factors of $n$.  We have
(derived in a fashion similar to the proof of the classical case over $\mq$ \cite[Theorem 2.10]{MVa1})
\begin{align}
\label{omegabound}
   \mathcal{W}_K(h) \ll \frac {\log N(h)}{\log \log N(h)}, \quad \mbox{for} \quad N(h) \geq 3.
\end{align}

   The above, together with \eqref{Lderbound} implies, that under GRH, for $\Re(s) \geq 1/2+\varepsilon$,
\begin{align}
\label{Lderboundgen}
  -(s-1) \frac {L'(s, \chi^{(m)}_j)}{L(s, \chi^{(m)}_j)}  \ll |s-1|\big((N(m)+2)(1+|s|)\big)^{\varepsilon}.
\end{align}

 Next, \cite[Theorem 5.19]{iwakow} asserts that when $\Re(s) \geq 1/2+\varepsilon$, under GRH,
\begin{align}
\label{PgLest1}
\begin{split}
& L( s,  \widehat\chi^{(m)}_j )^{-1} \ll |sN(m)|^{\varepsilon}.
\end{split}
\end{align}

  Applying \eqref{omegabound}, we see that
\begin{align}
\label{Lninvbound}
\begin{split}
 \Big | \prod_{\substack { (\varpi) \\ \varpi | m}}(1-\widehat \chi^{(m)}_j(\varpi)N(\varpi)^{-s})^{-1} \Big | \ll
N(m)^{\varepsilon}, \quad \Re(s) \geq 1/2+\varepsilon.
\end{split}
\end{align}

From \eqref{PgLest1} and \eqref{Lninvbound}, under GRH, for $\Re(s) \geq 1/2+\varepsilon$,
\begin{align}
\label{PgLest2}
\begin{split}
& L( s,  \chi^{(m)}_j )^{-1} \ll |sN(m)|^{\varepsilon}.
\end{split}
\end{align}

    Similarly, we derive from \cite[Theorem 5.19, Corollary 5.20]{iwakow} that under the generalized Lindel\"of hypothesis (a consequence of GRH), for $\Re(s) \geq 1/2$,
\begin{align}
\label{Lboundgen}
  (s-1)L(s, \chi^{(m)}_j) \ll |s-1|\big((N(m)+2)(1+|s|)\big)^{\varepsilon}.
\end{align}

  The above arguments apply to $L(s, \chi_{j, \varpi})$ for any $\varpi \equiv 1 \pmod 9$ as well, so that
\begin{align}
\label{Lboundgen1}
\begin{split}
  L( s,  \chi_{j,\varpi} )^{-1} \ll & |sN(\varpi)|^{\varepsilon}, \quad \Re(s) \geq 1/2 + \varepsilon, \\
  L(s, \chi_{j,\varpi}) \ll & |s-1|\big((N(m)+2)(1+|s|)\big)^{\varepsilon}, \quad \Re(s) \geq 1/2+\varepsilon.
\end{split}
\end{align}

\subsection{Analytic behavior of Dirichlet series associated with cubic Gauss sums}
\label{section: smooth Gauss}
  For any ray class character $\psi \pmod {9}$ and any primary $r$, we define
\begin{align}
\label{h}
   h(r,s;\psi)=\sum_{\substack{\varpi \text{ primary} \\ (\varpi,r)=1}}\frac {\Lambda_K(\varpi)\psi(\varpi)g_{K,j}(r,\varpi)}{N(\varpi)^s}.
\end{align}

   In this section, we establish certain analytic properties of $h(r,s;\psi)$ needed in this paper. To this end, we first quote the following result of
C. David and A. M. G\"{u}lo\u{g}lu in \cite[Theorem 4.4]{DG22}.
\begin{lemma}
\label{lemg3} Let $\psi$ be any ray class character $\pmod 9$ and let $r$ be any primary element in $K$.
We have for $x>1$ and any $\varepsilon>0$,
\begin{align}
\label{glambdabound}
\begin{split}
   \sum_{\substack {\varpi \odd \\ (\varpi, r)=1 \\ N(\varpi) \leq x}} \frac {\Lambda_K(\varpi)\psi(\varpi) g_{K,j}(r, \varpi) }{\sqrt{N(\varpi)}}  \ll & (xN(r))^{\varepsilon}\min (x, N(r)^{1/12}x^{5/6}).
\end{split}
\end{align}
\end{lemma}
\begin{proof}
   We first note that using the estimation that $g_{K, j}(r,\varpi) \ll N(\varpi)^{1/2}$ for any $(r, \varpi)=1$ by \eqref{2.7} and \cite[Proposition 7.5]{Lemmermeyer} to see that upon summing trivially, the left-hand side expression in \eqref{glambdabound} is $\ll (xN(r))^{\varepsilon}x$. Notice further that $x \leq N(r)^{1/12}x^{5/6}$ while $N(r) \geq x^{2}$. We may thus assume that $x^2 > N(r)$, in which case we observe that in \cite[Theorem 4.4]{DG22}, it is shown that the left-hand side expression in \eqref{glambdabound} is bounded by
\begin{align*}
%%\label{glambdabound}
\begin{split}
  x^{\varepsilon}(N(r)^{1/4+\varepsilon}x^{1/2}+N(r)^{1/6+\varepsilon}x^{2/3}+N(r)^{1/10+\varepsilon}x^{4/5}+N(r)^{1/12+\varepsilon}x^{5/6}).
\end{split}
\end{align*}
  Now one checks that when $x^2 > N(r)$, we have
\begin{align*}
%%\label{glambdabound}
\begin{split}
  N(r)^{1/4}x^{1/2} \leq N(r)^{1/6}x^{2/3} \leq  N(r)^{1/10}x^{4/5} \leq N(r)^{1/12}x^{5/6}.
\end{split}
\end{align*}
   From which the estimation given in \eqref{glambdabound} follows. This completes the proof of Lemma \ref{lemg3}. 
\end{proof}
   With the aid of Lemma \ref{lemg3}, we now establish the following result concerning analytic properties of $h(r,s;\psi)$.
\begin{lemma}
\label{lemma:laundrylist}
  With the notation as above, let $h(r,s;\psi)$ be as in \eqref{h}.  For any fixed $r$, the function $h(r,s;\psi)$ is analytical in the region when $3/2-1/6<\Re(s) \leq 3/2$.
Moreover, in this region,  we have,
\begin{align}
\label{hbound}
 & h(r,s;\psi) \ll N(r)^{1/2(3/2-s)+\varepsilon}.
\end{align}
\end{lemma}
\begin{proof}
  We recast the function $h(r,s;\psi)$ in \eqref{h} as
\begin{align}
\label{hdecomp}
   h(r,s;\psi)=\sum_{\substack{\varpi \text{ primary} \\ (\varpi,r)=1 \\ N(\varpi) \leq N(r)^{1/2} }}\frac {\Lambda_K(\varpi)\psi(\varpi)g_{K,j}(r,\varpi)}{N(\varpi)^s}+\sum_{\substack{\varpi \text{ primary} \\ (\varpi,r)=1 \\ N(\varpi) > N(r)^{1/2} }}\frac {\Lambda_K(\varpi)\psi(\varpi)g_{K,j}(r,\varpi)}{N(\varpi)^s}.
\end{align}
   We bound the first sum above trivially by
\begin{align}
\label{firstsum}
   \ll N(r)^{1/2(3/2-s)+\varepsilon}, \; \mbox{for} \; \Re(s) \leq 3/2,
\end{align}
using the estimation $g_{K,j}(r,\varpi) \ll N(\varpi)^{1/2}$ again. \newline

  We next estimate the second sum on the right-hand side of \eqref{hdecomp}  using \eqref{glambdabound} and partial summation to see that we get that for $\Re(s) > 3/2-1/6$, the sum in question is
\begin{align}
\label{secondsum}
   \ll N(r)^{1/12}N(r)^{1/2(3/2-1/6-s)+\varepsilon}.
\end{align}
 The estimation given in \eqref{hbound} now follows from \eqref{hdecomp}--\eqref{secondsum}. This completes the proof.
\end{proof}

\subsection{Some results on multivariable complex functions} Here, we quote some results from multivariable complex analysis. We begin with the notation of a tube domain.
\begin{defin}
		An open set $T\subset\mc^n$ is a tube if there is an open set $U\subset\mr^n$ such that $T=\{z\in\mc^n:\ \Re(z)\in U\}.$
\end{defin}
	
   For a set $U\subset\mr^n$, we define $T(U)=U+i\mr^n\subset \mc^n$.  We shall make use of the following Bochner's Tube Theorem \cite{Boc}.
\begin{theorem}
\label{Bochner}
		Let $U\subset\mr^n$ be a connected open set and $f(z)$ a function holomorphic on $T(U)$. Then $f(z)$ has a holomorphic continuation to the convex hull of $T(U)$.
\end{theorem}

 We denote the convex hull of an open set $T\subset\mc^n$ by $\widehat T$.  Our next result is \cite[Proposition C.5]{Cech1} on the modulus of holomorphic continuations of multivariable complex functions.
\begin{prop}
\label{Extending inequalities}
		Assume that $T\subset \mc^n$ is a tube domain, $g,h:T\rightarrow \mc$ are holomorphic functions, and let $\tilde g,\tilde h$ be their holomorphic continuations to $\widehat T$. If  $|g(z)|\leq |h(z)|$ for all $z\in T$ and $h(z)$ is nonzero in $T$, then also $|\tilde g(z)|\leq |\tilde h(z)|$ for all $z\in \widehat T$.
\end{prop}

\section{Proof of Theorem \ref{Theorem for all characters}}
\label{sec: mainthm}

\subsection{Setup}

  We define for $\Re(s), \Re(w)$ and $\Re(z)$ large enough,
\begin{align}
 \label{Aswzexp}
\begin{split}
A_{K,j}(s,w,z)=& \sumstar_{\varpi}\frac{\Lambda_{K}(\varpi) L(w,  \chi_{j, \varpi})}{N(\varpi)^sL(z,  \chi_{j, \varpi})}=\sumstar_{ n }\frac{\Lambda_{K}(n) L(w, \chi_{j, n})}{N(n)^sL(z,  \chi_{j, n})}-\sum_{l \geq 2}\sumstar_{\substack{\varpi^l \\ \varpi \odd}}\frac{\Lambda_{K}(\varpi^l) L(w, \chi_{j,\varpi^l})}{N(\varpi)^{ls}L(z,  \chi_{j, \varpi^l})} \\
:= & A_{K,j,1}(s,w,z)-A_{K,j,2}(s,w,z).
\end{split}
\end{align}

The Mellin inversion yields that
\begin{equation}
\label{Integral for all characters}
	\sumstar_{\substack{\varpi}}\frac {\Lambda_{K}(\varpi)L(\tfrac 12+\alpha, \chi_{j, \varpi})}{L(\tfrac 12+\beta, \chi_{j, \varpi})}w \bfrac {N(\varpi)}X=\frac1{2\pi i}\int\limits_{(c)}A_{K,j}\lz s,\tfrac12+\alpha, \tfrac12+\beta\pz X^s\widehat w(s) \dif s,
\end{equation}
where that the Mellin transform $\widehat{w}$ of $w$ is given by
\begin{align*}
     \widehat{w}(s) =\int\limits^{\infty}_0w(t)t^s\frac {\dif t}{t}.
\end{align*}

\subsection{Analytical properties of $A_{K,j}(s,w,z)$}
\label{sec: Aproperty}

  First note that we have
\begin{align}
\label{Aswmain-0}
\begin{split}
 & A_{K,j,1}(s,w,z)=\sumstar_{ n }\frac{\Lambda_{K}(n) L(w, \chi_{j, n})}{N(n)^sL(z, \chi_{j, n})}=\sumstar_{ n } \sum_{\substack{(m), (k)}}\frac{\Lambda_{K}(n) \mu_K(k)\chi^{(mk)}_j(n)}{N(m)^wN(k)^zN(n)^s}.
\end{split}
\end{align}
  We write $m=(1-\omega)^{r_1}m', k=(1-\omega)^{r_2}k'$ with $r_1, r_2 \geq 0$, $m', k'$ primary and further replacing $m', k'$ by $m, k$.   By \eqref{cubicsupplyment}, we get that
\begin{align}
\label{Aswmainsimplified}
\begin{split}
 & \sumstar_{ n } \sum_{\substack{(m), (k)}}\frac{\Lambda_{K}(n) \mu_K(k)\chi^{(mk)}_j(n)}{N(m)^wN(k)^zN(n)^s}=\sum_{\substack{r_1,r_2 \geq 0  }}\frac{\mu_K((1-\omega)^{r_2})}{3^{r_1w}3^{r_2z}}\sum_{\substack{m,k \odd }}\frac{\mu_K(k)}{N(m)^wN(k)^z}\sumstar_{ n }\frac{\Lambda_{K}(n)\chi^{(mk)}_j(n)}{N(n)^s}.
\end{split}
\end{align}

We further apply the ray class characters to detect the condition that $n \equiv 1 \pmod {9}$ to recast the right-hand side expression above as
\begin{align}
\label{Aswmain-1}
\begin{split}
\frac{1}{\# h_{(9)}} \sum_{\substack{r_1,r_2 \geq 0  }} & \frac{\mu_K((1-\omega)^{r_2})}{3^{r_1w}3^{r_2z}} \sum_{\psi \bmod {9}}\sum_{\substack{m,k \\ \odd }}\frac{\mu_K(k)}{N(m)^wN(k)^z}\sum_{ n \odd }\frac{\Lambda_{K}(n)\psi(n)\chi^{(mk)}_j(n)}{N(n)^s} \\
=&  - \frac{1}{\# h_{(9)}} \sum_{\substack{r_1,r_2 \geq 0  }}\frac{\mu_K((1-\omega)^{r_2})}{3^{r_1w}3^{r_2z}}\sum_{\psi \bmod {9}}\sum_{\substack{m,k \\ \odd }}\frac{\mu_K(k)}{N(m)^wN(k)^z} \frac {L'(s, \psi \chi^{(mk)})}{L(s, \psi \chi^{(mk)})}.
\end{split}
\end{align}

 Note $A_{K,j,2}(s,w,z)$ is easily seen, using \eqref{Lboundgen1}, to converge when $\Re(s)>1/2$, $\Re(w) \geq 1/2$, $\Re(z)>1/2$ under GRH. We deduce from this, \eqref{Lderboundgen} and \eqref{Aswmain-1} that under GRH, save for a simple pole at $s=1$ when $mk$ is a perfect cube and $\psi$ is the principal character, $A_{K,j}(s,w,z)$ is convergent in the region when $\Re(s)>1/2$, $\Re(w)>1$, $\Re(z)>1$.  On the other hand, we observe that the sum over $\varpi$ in \eqref{Aswmain-0} is convergent when $\Re(s)>1$, $\Re(w)\geq 1/2$, $\Re(z)>1/2$ under GRH by \eqref{PgLest2} and \eqref{Lboundgen}. It therefore follows the above and Theorem \ref{Bochner} that the left-hand side expression in \eqref{Aswmain-0} is convergent in the region
\begin{equation}
\label{S1}
		S_{1} :=\{(s,w,z): \ \Re(s)> \tfrac 12, \ \Re(w) \geq \tfrac 12, \ \Re(z)>\tfrac 12, \ \Re(s+w)> \tfrac 32, \ \Re(s+z)> \tfrac 32 \}.
\end{equation}

   In the sequal, we shall define similar regions $S_j, S_{i, j}$ and we adopt the convention that for any real number $\delta$,
\begin{align*}
%%\label{Aswboundwlarge}
\begin{split}
 S_{j,\delta} :=& \{ (s,w,z)+\delta (1,1,1) : (s,w,z) \in S_j,  \ \Re(w) \leq 1 \}, \\
S_{i, j,\delta} :=& \{ (s,w,z)+\delta (1,1,1) : (s,w,z) \in S_{i,j}, \ \Re(w) \leq 1 \} .
\end{split}
\end{align*}

  We also deduce from \eqref{Lderboundgen}, \eqref{PgLest2}, \eqref{Lboundgen} and apply Proposition \ref{Extending inequalities} to see that under GRH, in the region $S_{1,\varepsilon}$, we have for any $\varepsilon>0$,
\begin{align}
\label{Aswboundwlarge6}
\begin{split}
 \Big|(s-1)A_{K,j}(s,w,z) \Big | \ll & (1+|s|)^{1+\varepsilon}(1+|w|)^{\varepsilon}|z|^{\varepsilon}.
\end{split}
\end{align}

  Lastly, we apply the functional equation \eqref{fneqnL} of $L(w, \chi_{j, \varpi})$ to deduce from \eqref{Aswzexp} that
\begin{align*}
%%\label{Aswfuneqn}
\begin{split}
A_{K,j}(s,w, z)= \sumstar_{\varpi}\frac{\Lambda_{K}(\varpi) L(w,  \chi_{j, \varpi})}{N(\varpi)^sL(z,  \chi_{j, \varpi})} = |D_K|^{1/2-w}(2\pi)^{2w-1}\frac {\Gamma(1-w)}{\Gamma (w)} \sumstar_{\varpi} \frac{\Lambda_{K}(\varpi)g_{K,j}(\varpi) L(1-w,  \overline \chi_{j, \varpi})}{N(\varpi)^{s+w}L(z,  \chi_{j, \varpi})}.
\end{split}
\end{align*}

   Note that by \eqref{2.7},
\begin{align*}
%%\label{Aswfuneqninterchange}
\begin{split}
 \sumstar_{\varpi} & \frac{\Lambda_{K}(\varpi)g_{K,j}(\varpi) L(1-w,  \overline \chi_{j, \varpi})}{N(\varpi)^{s+w}L(z,  \chi_{j, \varpi})} = \sumstar_{\varpi} \sum_{\substack{ m,k \\ \odd}} \frac{\Lambda_{K}(\varpi)\mu_K(k)g_{K,j}(\varpi) \chi_{j, \varpi}(k) \overline \chi_{j, \varpi}(m)}{N(\varpi)^{s+w}N(m)^{1-w}N(k)^{z}} \\
& \hspace*{1.5cm} = \sum_{\substack{ m,k \\ \odd}} \frac {\mu_K(k)}{N(m)^{1-w}N(k)^{z}} \sumstar_{\substack{\varpi \odd \\ (\varpi, mk)=1}}  \frac{\Lambda_{K}(\varpi)g_{K,j}(\varpi) \overline \chi_{j, \varpi}(mk^2)}{N(\varpi)^{s+w}} \\
& \hspace*{1.5cm} = \sum_{\substack{ m,k \\ \odd}} \frac {\mu_K(k)}{N(m)^{1-w}N(k)^{z}} \frac{1}{\# h_{(9)}} \sum_{\psi \bmod {9}} \ \sum_{\substack{\varpi \odd \\ (\varpi, mk)=1}}   \frac{\Lambda_{K}(\varpi)\psi(\varpi)g_{K,j}(mk^2, \varpi) }{N(\varpi)^{s+w}}.
\end{split}
\end{align*}

   We now apply Lemma \ref{lemma:laundrylist} to the inner sum in the last expression above to deduce that
the last triple sum above and hence $A_{K,j}(s,w, z)$ is convergent in the region
\begin{equation}
\label{S2}
		S_{2} :=\{(s,w,z): \ \tfrac 32-\tfrac 16< \Re(s+w) \leq \tfrac 32, \ \Re(1-w+\tfrac 12(s+w)) >\tfrac 74, \ \Re(z+s+w)>\tfrac 52 \}.
\end{equation}

 Moreover, in the region $S_{2, \varepsilon}$, we have by \eqref{Stirlingratio} that
\begin{align}
\label{Aswboundwlarge7}
\begin{split}
 \Big|A_{K,j}(s,w, z) \Big | \ll  (1+|w|)^{1-2\Re(w)+\varepsilon}.
\end{split}
\end{align}

  Note that the union of $S_1$ and $S_2$ is connected and we determine the convex hull of $S_1$ and $S_2$ by rewriting the condition $\Re(z+s+w)>5/2$ as $\Re(s+w)> 5/2-\Re(z)$ to note that $5/2-\Re(z) \geq 3/2- 1/6$ when $\Re(z) \leq 7/6$. We now recast the condition $\Re(1-w+(s+w)/2) >7/4$ as $\Re(s-w) > 3/2$ to note that the intersection of the three planes $\Re(s+w)=5/2-\Re(z)$, $\Re(s-w) = 3/2$, $\Re(z)= 1/2$ is at $(7/4, 1/4, 1/2)$ while the intersection of the three planes $\Re(s+w)=5/2-\Re(z), \Re(s-w) =3/2$, $\Re(z)=7/6$ is at $(17/12, -1/12 , 7/6)$. These two points are on the boundary of $S_2$. As the point $(1, 1/2, 1/2)$ is on the boundary of $S_1$, it together with the two points $(7/4, 1/4 , 1/2),  (17/12, -1/12 , 7/6)$
determines a plane whose equation is given by $\Re(s+3w+2z)=7/2$. When $\Re(z) \geq 7/6$, we notice that the line determined by the intersection of the two planes $\Re(s+w)= 3/2$, $\Re(s+z) = 3/2$ meets the plane $\Re(w)=1/2$ at the point $(1, 1/2, 1/2)$ and meets the plane $\Re(s)=1/2$
at the point $(1/2, 1, 1)$. These two points together with the point $(17/12, -1/12, 7/6)$ determines a plane whose equation is given by $\Re(15s+13w+2z)=45/2$. Lastly, the points $(17/12, -1/12 , 7/6)$, $(1/2, 1, 1)$ determine a plane that is parallel to the $z$-axis, whose equation is given by $\Re(w+13/11(s-1/2))=1$. It follows that the convex hull of $S_1$ and $S_2$ equals
\begin{align}
\label{S3}
\begin{split}
		S_3 :=& \{(s,w, z): \ \Re(s)> \tfrac 12, \ \Re(z)>\tfrac 12, \ \Re(s+z)> \tfrac 32, \ \Re(s+w)> \tfrac 32-\tfrac 16,  \\
& \hspace{2cm}  \ \Re(s+3w+2z)> \tfrac 72, \ \Re(15s+13w+2z)> \tfrac {45}{2}, \ \Re(w+\tfrac {13}{11}(s-\tfrac 12))>1 \}.
\end{split}
\end{align}
  It follows from Theorem \ref{Bochner} that $A_{K,j}(s,w,z)$ converges absolutely in the region $S_3$. Furthermore, we deduce from Proposition~\ref{Extending inequalities} the bound, inherited from \eqref{Aswboundwlarge6} and \eqref{Aswboundwlarge7}, that in the region $S_{3,\varepsilon} \cap \{ (s,w, z)| \Re(w) > \frac 12-\frac 1{11} \}$, we have
\begin{align}
\label{Aswboundslarge2}
\begin{split}
 |(s-1)A_{K,j}(s,w,z)| \ll &
\begin{cases}
(1+|s|)^{1+\varepsilon}(1+|w|)^{2/11+\varepsilon}|z|^{\varepsilon}, & 1/2-  1/11< \Re(w) < 1/2, \\
(1+|s|)^{1+\varepsilon}(1+|w|)^{\varepsilon}|z|^{\varepsilon}, &  \Re(w) \geq 1/2.
\end{cases}
\end{split}
\end{align}

\subsection{Residue}
\label{sec:resA}

  We see from \eqref{Aswmainsimplified} that $A_{K,j}(s,w,z)$ has a simple pole at $s=1$ arising from the terms with $mk$ being perfect cubes and $\psi$ being the principal character.  As the residue of $-\zeta'_K(s)/\zeta_K(s)$ at $s=1$ equals $1$, the residue at $s=1$ of the left-hand side expression in \eqref{Aswmainsimplified} equals
\begin{align}
\label{Ress=1}
\begin{split}
 \frac{1}{\# h_{(9)}} & \sum_{\substack{r_1,r_2 \geq 0  }}\frac{\mu_K((1-\omega)^{r_2})}{3^{r_1w}3^{r_2z}}\sum_{\substack{m,k \odd \\ mk=\text{a perfect cube} }}\frac{\mu_K(k)}{N(m)^wN(k)^z} \\
 =& \frac{1}{\# h_{(9)}} \frac{1-3^{-z}}{1-3^{-w}} \prod_{(\varpi, j)=1}\Big(1+(1-\frac 1{N(\varpi)^{z-w}})\frac 1{N(\varpi)^{3w}}(1-\frac 1{N(\varpi)^{3w}} )^{-1}\Big ) = \frac{1}{\# h_{(9)}}  \frac{1-3^{-z}}{1-3^{-w}} \cdot \frac {\zeta^{(j)}_K(3w)}{\zeta^{(j)}_K(2w+z)}.
\end{split}
\end{align}

\subsection{Completion of the proof}
\label{sec: conclusion}

 We shift the line of integration in \eqref{Integral for all characters} to $\Re(s)=E(\alpha, \beta)+\varepsilon$.  The integral on the new line can be absorbed into the $O$-term in \eqref{Asymptotic for ratios of all characters} upon using \eqref{Aswboundslarge2} and the observation that repeated integration by parts gives that, for any integer $E \geq 0$,
\begin{align*}
%%\label{whatbound}
 \widehat w(s)  \ll  \frac{1}{(1+|s|)^{E}}.
\end{align*}
   We also encounter a simple pole at $s=1$ in this process with the corresponding residue given in \eqref{Ress=1}.  Direct computations now lead to the main terms given in \eqref{Asymptotic for ratios of all characters}. This completes the proof of Theorem \ref{Theorem for all characters}.

\section{Proof of Theorem \ref{Theorem for log derivatives}}
\label{sec: logder}

	We first recast the expression in \eqref{Asymptotic for ratios of all characters} as
\begin{align}
\label{Asymptoticshortversion}
\begin{split}	
& \sumstar_{\substack{\varpi}}\frac {\Lambda_{K}(\varpi)L(\tfrac 12+\alpha, \chi_{j, \varpi})}{L(\tfrac 12+\beta, \chi_{j, \varpi})}w \bfrac {N(\varpi)}X =   XM(\alpha,\beta)+R(X,\alpha,\beta),
\end{split}
\end{align}
   where $R(X,\alpha,\beta)$ is the error term in \eqref{Asymptotic for ratios of all characters} and we define
\begin{align*}
%%\label{M1M2}
\begin{split}
	M(\alpha,\beta)=& \frac {\M w(1)}{\# h_{(9)}} \frac{1-3^{-1/2-\beta}}{1-3^{-1/2-\alpha}}
\frac {\zeta^{(j)}_K(\tfrac{3}{2}+3\alpha)}{\zeta^{(j)}_K( \tfrac{3}{2}+2\alpha+\beta)} .
\end{split}
\end{align*}
	Note that the expression on the left-hand side of \eqref{Asymptoticshortversion} and $M(\alpha,\beta)$ are analytic functions of $\alpha,\beta$, so is $R(X,\alpha,\beta)$. \newline
	
	We now differentiate the above terms with respect to $\alpha$ for a fixed $\beta=r$ with $\varepsilon <\Re(r) <1/2$, and then set $\alpha=\beta=r$.  We get that
\begin{align}
\label{M1der}
\begin{split}
			\frac{\dif}{\dif \alpha} XM(\alpha,\beta)\Big\vert_{\alpha=\beta=r}
			&=\frac {\M w(1)X}{\# h_{(9)}}\lz\frac{(\zeta^{(j)}_K(\frac 32+3r))'}{\zeta^{(j)}_K(\frac 32+3r)}-\frac{\log 3 }{3^{1/2+r}-1}\pz.
\end{split}
\end{align}

	Next, as $R(X,\alpha,\beta)$ is analytic in $\alpha$, Cauchy's integral formula leads to
\begin{align*}
%%\label{Rder}
\begin{split}
		\frac{\dif}{\dif \alpha}R(X,\alpha,\beta)=\frac{1}{2\pi i}\int\limits_{C_\alpha}\frac{R(X,z,\beta)}{(z-\alpha)^2} \dif z,
\end{split}
\end{align*}
   where $C_\alpha$ is a circle centered at $\alpha$ of radius $\rho$ with $\varepsilon/2<\rho<\varepsilon$. It follows that
\begin{align*}
%%\label{Rderest}
\begin{split}
		\lab\frac{\dif }{\dif \alpha}R(X,\alpha,\beta)\rab\ll \frac1{\rho}\cdot\max_{z\in C_\alpha} |R(X,z,\beta)|\ll (1+|\alpha|)^{\varepsilon}(1+|\beta|)^{\varepsilon}X^{E(\alpha,\beta)+\varepsilon}.
\end{split}
\end{align*}

    We now set $\alpha=\beta=r$ to deduce from \eqref{Nab} and the above that
\begin{align}
\label{Rder1}
\begin{split}
		\lab\frac{\dif}{\dif \alpha}R(X,\alpha,\beta)\rab\ll (1+|r|)^{\varepsilon}X^{1-\Re(r)+\varepsilon}.
\end{split}
\end{align}
	
   We conclude from \eqref{Asymptoticshortversion}--\eqref{Rder1} that \eqref{Sum of L'/L with removed 2-factors} holds. This completes the proof of Theorem \ref{Theorem for log derivatives}.
	
\section{Proof of Theorem \ref{Theorem one-level density}}
\label{Section one-level density}
	
   We apply the residue theorem to see that
\begin{align}
\label{sumoverzeros}
\begin{split}
		\sum_{\gamma_n}h\bfrac{\gamma_{\varpi,n}\log X}{2\pi}=\frac{1}{2\pi i}\lz \ \int\limits_{(2)}-\int\limits_{(-1)} \ \pz h\lz\frac{\log X}{2\pi i}\left( s- \frac{1}{2} \right) \pz\frac {\Lambda'}{\Lambda}(s, \chi_{j, \varpi}) \dif s.
\end{split}
\end{align}

   We then deduce from the functional equation  \eqref{fneqn} that
\begin{align}
\label{Lambdarel}
\begin{split}
		\frac {\Lambda'}{\Lambda}(s, \chi_{j, \varpi})=-\frac {\Lambda'}{\Lambda}(1-s, \overline \chi_{j, \varpi}).
\end{split}
\end{align}

   It follows from \eqref{Lambda}, \eqref{sumoverzeros} and \eqref{Lambdarel} that
\begin{align*}
%%\label{sumoverzeros1}
\begin{split}
		\sum_{\gamma_n}h\bfrac{\gamma_n\log X}{2\pi}=&\frac{1}{2\pi i}\int\limits_{(2)}h\lz\frac{\log X}{2\pi i} \left( s- \frac{1}{2} \right)  \pz \lz \frac {\Lambda'}{\Lambda}(s, \chi_{j, \varpi})+\frac {\Lambda'}{\Lambda}(s, \overline \chi_{j, \varpi}) \pz \dif s \\
=& \frac{1}{2\pi i}\int\limits_{(2)}h\lz\frac{\log X}{2\pi i} \left( s- \frac{1}{2} \right) \pz \lz \frac {L'}{L}(s, \chi_{j, \varpi})+\frac {L'}{L}(s, \overline \chi_{j, \varpi})+2\frac {\Gamma'}{\Gamma}(s)+\log \frac {|D_K|N(\varpi)}{4\pi^2} \pz \dif s.
\end{split}
\end{align*}

  We conclude from the above that
\begin{align}
\label{Dsimplified}
\begin{split}
	F(X)D(X;h) = \frac{1}{2\pi i}\int\limits_{(2)} & h\lz\frac{\log X}{2\pi i}(s-1/2)\pz \sumstar_{\substack{\varpi \odd}}\Lambda_K(\varpi) w \bfrac {N(\varpi)}X \\
& \hspace{0.1in} \times \lz \frac {L'}{L}(s, \chi_{j, \varpi})+\frac {L'}{L}(s, \overline \chi_{j, \varpi})+2\frac {\Gamma'}{\Gamma}(s)+\log \frac {|D_K|N(\varpi)}{4\pi^2} \pz \dif s.
\end{split}
\end{align}

From \cite[Lemma 10.1]{Cech1}, if the Fourier transform of $h$ is supported in the interval $[-a,a]$, then
\begin{equation}
\label{Lemma size of h}
		h\bfrac{s\log X}{2\pi i}\ll \frac{X^{a\cdot\Re(s)}}{|s|^2(\log X)^2}.
\end{equation}

  We now evaluate
\begin{equation*}
	I:= \frac{1}{2\pi i}\int_{(2)}h\lz\frac{\log X}{2\pi i}(s-1/2)\pz\sumstar_{\substack{\varpi \odd}}\Lambda_K(\varpi) w \bfrac {N(\varpi)}X\lz \frac {L'}{L}(s, \chi_{j, \varpi})+\frac {L'}{L}(s, \overline \chi_{j, \varpi}) \pz ds
\end{equation*}
  by noting that $L(s, \overline \chi_{j, \varpi})=\overline{L(\overline s, \chi_{j, \varpi})}$ for any primary prime $\varpi$, so that by \eqref{Sum of L'/L with removed 2-factors},
\begin{align}
\label{Sum of L'/L with removed 2-factors1}
\sumstar_{\substack{\varpi}}\frac {\Lambda_{K}(\varpi)L'(\tfrac 12+r, \overline \chi_{j, \varpi})}{L(\tfrac 12+r, \overline \chi_{j, \varpi})}w \bfrac {N(\varpi)}X
= \frac {\M w(1)X}{\# h_{(9)}}\lz\frac{(\zeta^{(j)}_K(\tfrac{3}{2}+3 r))'}{\zeta^{(j)}_K(\tfrac{3}{2}+3 r)}-\frac{\log 3 }{3^{1/2+r}-1}\pz+O((1+|r|)^{\varepsilon}X^{1-\Re(r)+\varepsilon}).
\end{align}

 We set $s=1/2+r+it$ with $0<r<1/2$ and apply \eqref{Sum of L'/L with removed 2-factors},  \eqref{Sum of L'/L with removed 2-factors1} to see that
\begin{align}
\label{Integral after ratios}
\begin{split}
	I=&	\frac {\M w(1)X}{\pi \# h_{(9)}}\int_{-\infty}^{\infty} h\lz\frac{\log X}{2\pi i}(r+it)\pz \lz\frac{(\zeta^{(j)}_K( \tfrac{3}{2}+3r+3it))'}{\zeta^{(j)}_K(\tfrac{3}{2}+3r+3it)}-\frac{\log 3 }{3^{1/2+r+it}-1}\pz dt \\
& +\frac{1}{2\pi i}\int_{(2)}h\lz\frac{\log X}{2\pi i} \left( s- \frac{1}{2} \right)\pz O(|t|^{\varepsilon}X^{1-r+\varepsilon}) dt.
\end{split}
\end{align}
  We bound the last integral above using \eqref{Lemma size of h} to see that it is
\begin{align*}
%%\label{Error in one-level density}
  \ll & X^{1-r+\varepsilon}\int\limits_{-\infty}^{\infty}h\lz\frac{\log X}{2\pi i}(r+it)\pz|t|^{\varepsilon} \dif t \ll X^{1-r+ar+\varepsilon}.
\end{align*}
  We then set $r=1/2-\varepsilon$ to obtain the error term in \eqref{Onelevel}. \newline

 We also evaluate the first expression on the right-hand side of \eqref{Integral after ratios} by shifting the line of integration to $r=0$ to see that
\begin{align}
\label{Main term in one-level density}
\begin{split}
	I=	 \frac {\M w(1) X}{\pi \# h_{(9)}} & \int\limits_{-\infty}^{\infty}h\lz\frac{\log X}{2\pi}t\pz \lz\frac{(\zeta^{(j)}_K(\tfrac{3}{2}+3it))'}{\zeta^{(j)}_K(\tfrac{3}{2}+3it)}-\frac{\log 3 }{3^{1/2+it}-1}\pz \dif t\\
		=& \frac {2 \M w(1) }{\# h_{(9)}} \cdot \frac{X}{\log X}\int\limits_{-\infty}^{\infty}h(u)\lz\frac{(\zeta^{(j)}_K(\tfrac{3}{2}+ \frac {6\pi iu}{\log X}))'}{\zeta^{(j)}_K(\tfrac{3}{2}+\frac {6\pi iu}{\log X})}-\frac{\log 3 }{3^{1/2+ 2\pi iu/\log X-1}}\pz \dif u.
\end{split}
\end{align}

   We next compute
\begin{align*}
%%\label{Dsimplified1}
\begin{split}
	\frac{1}{2\pi i}\int_{(2)}h\lz\frac{\log X}{2\pi i} \left( s- \frac{1}{2} \right) \pz\sumstar_{\substack{\varpi \odd}}\Lambda_K(\varpi) w \bfrac {N(\varpi)}X\lz 2\frac {\Gamma'}{\Gamma}(s)+\log \frac {|D_K|N(\varpi)}{4\pi^2} \pz ds
\end{split}
\end{align*}
  by shifting the line of integration to $\Re(s)=1/2$ to see that it equals
\begin{align}
\label{Dsimplified2}
\begin{split}
\frac {2\widehat h(1)}{\log X} \sumstar_{\substack{\varpi \odd}}\Lambda_K(\varpi) w \bfrac {N(\varpi)}X\log \frac {|D_K|N(\varpi)}{4\pi^2}
 + \sumstar_{\substack{\varpi \odd}}\Lambda_K(\varpi) w \bfrac {N(\varpi)}X \frac{1}{\pi}\int^{\infty}_{-\infty}h\lz\frac{\log X}{2\pi}t \pz \frac {\Gamma'}{\Gamma}(\tfrac 12+it) dt.
\end{split}
\end{align}

 As $h$ is even, we see that
\begin{align}
\label{Dsimplified3}
\begin{split}
	\frac{1}{\pi}\int\limits^{\infty}_{-\infty}h\lz\frac{\log X}{2\pi}t \pz \frac {\Gamma'}{\Gamma}(\frac 12+it) dt=& \frac{1}{2\pi}\int\limits^{\infty}_{-\infty}h\lz\frac{\log X}{2\pi}t \pz \lz \frac {\Gamma'}{\Gamma}(\frac 12+it) +\frac {\Gamma'}{\Gamma}(\frac 12-it) \pz \dif t \\
=& \frac 1{\log X}\int\limits^{\infty}_{-\infty}h\lz u \pz \lz \frac {\Gamma'}{\Gamma}\left( \frac 12+\frac {2 \pi iu}{\log X} \right) +\frac {\Gamma'}{\Gamma}\left( \frac 12-\frac {2 \pi iu}{\log X} \right) \pz \dif u.
\end{split}
\end{align}

  The expression given in \eqref{Onelevel} now readily follows from \eqref{Size of family}, \eqref{Dsimplified}, \eqref{Main term in one-level density}--\eqref{Dsimplified3}. We further simplify the right-hand side expression in \eqref{Onelevel} by noticing that for any real $y \geq 0$,
\begin{align*}
%%\label{sumstarvarpi}
\begin{split}
			 \sumstar_{\substack{\varpi \odd \\ N(\varpi) \leq y}}\Lambda_K(\varpi) =\frac{1}{\# h_{(9)}} \sum_{\psi \bmod {9}} \ \sum_{\substack{\varpi \odd \\ N(\varpi) \leq y}}\Lambda_K(\varpi)\psi(\varpi).
\end{split}
\end{align*}

  It follows from \cite[(5.49)]{iwakow} and \cite[Theorem 5.15]{iwakow} that
\begin{align*}
%%\label{sumstarvarpitwist}
\begin{split}
			 \sum_{\substack{\varpi \odd \\ N(\varpi) \leq y}}\Lambda_K(\varpi)\psi(\varpi)=
\begin{cases}
  y+O(y^{1/2+\varepsilon}), \quad \text{$\psi$ is principal}, \\
  O(y^{1/2+\varepsilon}), \quad \text{otherwise}.
\end{cases}
\end{split}
\end{align*}

   We then conclude that
\begin{align}
\label{sumstarvarpiest}
\begin{split}
			 \sumstar_{\substack{\varpi \odd \\ N(\varpi) \leq y}}\Lambda_K(\varpi) =\frac{y}{\# h_{(9)}}+O(y^{1/2+\varepsilon}).
\end{split}
\end{align}

   The above estimation and partial summation now renders
\begin{align}
\label{Fest}
\begin{split}
		& F(X)= \int\limits^{\infty}_0  w \bfrac {u}X  \dif \left( \sumstar_{\substack{\varpi \odd \\ N(\varpi) \leq u}}\Lambda_K(\varpi) \right)= \frac{\widehat w(1)}{\# h_{(9)}}X+O(X^{1/2+\varepsilon}),  \\
 & \sumstar_{\substack{\varpi \odd}}\Lambda_K(\varpi) w \bfrac {N(\varpi)}X\log N(\varpi) = \frac{\widehat w(1)}{\# h_{(9)}}X\log X+\frac{1}{\# h_{(9)}}\left( \int\limits^{\infty}_0w(u)\log u \dif u \right) X+O(X^{1/2+\varepsilon}).
\end{split}
\end{align}

  Note that for any $a, b \in \mr$, we have the approximate formula (cf.~\cite[8.363.3]{GR})
\begin{align*}
%%\label{Dsimplified4}
\begin{split}
  \frac{\Gamma'}{\Gamma}(a + ib) + \frac{\Gamma'}{\Gamma}(a-ib) = 2\frac{\Gamma'}{\Gamma} (a) + O \left( \leg ba^2
\right).
\end{split}
\end{align*}

   It follows that
\begin{align}
\label{Dsimplified5}
\begin{split}
  \frac{1}{2\pi}\int^{\infty}_{-\infty}h\lz\frac{\log X}{2\pi}t \pz \lz \frac {\Gamma'}{\Gamma}(\frac 12+it) +\frac {\Gamma'}{\Gamma}(\frac 12-it) \pz dt
=&  \frac{4 }{\log X}\frac{\Gamma'}{\Gamma}(\frac 12) \widehat h(1) + O
((\log X)^{-3} ).
\end{split}
\end{align}

   We readily deduce \eqref{Onelevelasym} from \eqref{Onelevel}, \eqref{Fest} and \eqref{Dsimplified5}. This completes the proof of Theorem \ref{Theorem one-level density}.

\section{Proof of Theorems \ref{Theorem for all charactersQ}--\ref{Theorem one-level densityQ}}
\label{sec thirdthm}

  As the proof of Theorems \ref{Theorem for all charactersQ}-\ref{Theorem one-level densityQ} is similar to that of Theorems \ref{Theorem for all characters}-\ref{Theorem one-level density}, we shall only give a sketch of the proof of Theorem \ref{Theorem for all charactersQ} here.  We recall from \cite[Lemma 2.2]{G&Zhao2022-4} that the primitive cubic Dirichlet characters of prime conductor $p$ co-prime to $3$ are induced by
$\chi_{j,\varpi}$ for some prime $\varpi \in \mathcal O_K$ such that $N(\varpi) = p$. It follows that in this case there are exactly two such cubic Dirichlet characters,  induced by $\chi_{j,\varpi}$ and $\overline \chi_{j,\varpi}$. We then deduce that
\begin{align}
\label{ratioLQ}
\begin{split}
 \sum_{\substack{ p=N(\varpi) \\ \varpi \equiv 1 \bmod {9}} } & \;
\sum_{\substack{\chi \bmod{p} \\ \chi^3 = \chi_0, \ \chi \neq \chi_0}} \Lambda(p) \frac {L(\tfrac{1}{2}+\alpha, \chi)}{L(\tfrac{1}{2}+\beta, \chi)}  \Phi \leg{p}{Q} \\
=&
\sumstar_{\varpi}\Lambda_K(\varpi)\frac {L_{\mq}(\tfrac{1}{2}+\alpha, \chi_{j, \varpi})}{L_{\mq}(\tfrac{1}{2}+\beta, \chi_{j, \varpi})}  \Phi \leg{N(\varpi)}{Q} + \sumstar_{\varpi}\Lambda_K(\varpi) \frac {L_{\mq}(\tfrac{1}{2}+\alpha, \overline \chi_{j, \varpi})}{L_{\mq}(\tfrac{1}{2}+\beta, \overline \chi_{j, \varpi})} \Phi \leg{N(\varpi)}{Q},
\end{split}
\end{align}
  where $L_{\mq}(s, \chi_{j, \varpi})$ and $L_{\mq}(s, \overline \chi_{j, \varpi}))$ stand for the Dirichlet $L$-function associated to $\chi_{j, \varpi}$ and $\overline \chi_{j, \varpi}$, respectively,  by treating them as Dirichlet characters. \newline

  We define, for $\Re(s), \Re(w)$ and $\Re(z)$ large enough,
\begin{align}
 \label{Bswzexp}
\begin{split}
 B_{j}(s,w,z)=& \sumstar_{\varpi}\frac{\Lambda_{K}(\varpi) L_{\mq}(w,  \chi_{j, \varpi})}{N(\varpi)^sL_{\mq}(z,  \chi_{j, \varpi})}.
\end{split}
\end{align}

  Then Mellin inversion leads to
\begin{equation}
\label{IntegralB}
	\sumstar_{\substack{\varpi}}\frac {\Lambda_{K}(\varpi)L_{\mq}(\tfrac 12+\alpha, \chi_{j, \varpi})}{L_{\mq}(\tfrac 12+\beta, \chi_{j, \varpi})}\Phi \bfrac {N(\varpi)}Q=\frac1{2\pi i}\int\limits_{(c)}B_{j}\lz s,\tfrac12+\alpha, \tfrac12+\beta\pz X^s\widehat \Phi(s) \dif s,
\end{equation}

   We now obtain the analytical properties of $B_{j}(s,w,z)$ similar to what has been done in Section \ref{sec: Aproperty} for $A_{K,j}(s,w,z)$, upon making the following replacements: $\mu_K$ by the usual M\"obius function  $\mu$ on $\mz$, $1-\omega$ by $3$, the condition $m, k$ primary by $m, k \geq 1, (mk, 3)=1$, the norms $N(m), N(k)$ by $m, k$. This allows us to deduce that other than a simple pole at $s = 1$ when $mk$ is a perfect cube
and $\psi$ is the principal character, $B_{j}(s,w,z)$ is convergent in the region $S_1$ as well, where $S_1$ is defined in \eqref{S1} and an estimation similar to that given in \eqref{Aswboundwlarge6} with $A_{K,j}$ replaced by $B_K$ holds also.  Moreover, the residue at $s=1$ equals to
\begin{align}
\label{Ress=1Q}
\begin{split}
  \frac{1}{\# h_{(9)}} \frac{1-3^{-z}}{1-3^{-w}} \frac {\zeta^{(j)}(3w)}{\zeta^{(j)}(2w+z)}.
\end{split}
\end{align}

 Further, for any primitive Dirichlet character $\chi$ modulo $q$, let $\af=0$ or $1$ be given by $\chi(-1)=(-1)^{\af}$.  Set
\begin{align*}
  \Lambda(s, \chi)= \left( \frac {q}{\pi} \right)^{(s+\af)/2}\Gamma \left( \frac 12(s+\af) \right)L(s, \chi).
\end{align*}
   Then $\Lambda(s, \chi)$ extends to an entire function on $\mc$ when $\chi \neq \chi_0$ and satisfies the following functional equation (see \cite[Theorem 4.15]{iwakow}):
\begin{align*}
  \Lambda(s, \chi)=\frac {\tau(\chi)}{i^{\af}q^{1/2}}\Lambda(1-s, \overline \chi), \; \mbox{where} \;  \tau(\chi)=\sum_{x \bmod {n}}\chi(x)e(x)
\end{align*}
is the Gauss sum.\newline

  We apply the above to $\chi=\chi_{j, \varpi}$ for a prime $\varpi \equiv 1 \pmod 9$, regarded as a primitive cubic Dirichlet character modulo $N(\varpi)$. Notice that in this case we have ${\af}=0$ since $(-1)^3=(-1)$. Also, we have $\tau(\chi_{j, \varpi})=g_{K, j}(\varpi)$ by \cite[(2.2)]{G&Zhao2022-4}. It follows that
\begin{align}
\label{Bswfuneqn}
 B_{j}(s,w, z)= \sumstar_{\varpi}\frac{\Lambda_{K}(\varpi) L_{\mq}(w,  \chi_{j, \varpi})}{N(\varpi)^sL_{\mq}(z,  \chi_{j, \varpi})}
=  \pi^{w-1/2}\frac {\Gamma(\frac {1-w}{2})}{\Gamma (\frac w2)} \sumstar_{\varpi} \frac{\Lambda_{K}(\varpi)g_{K,j}(\varpi) L_{\mq}(1-w,  \overline \chi_{j, \varpi})}{N(\varpi)^{s+w}L_{\mq}(z,  \chi_{j, \varpi})}.
\end{align}

  We then conclude that $B_{j}(s,w, z)$ also converges in the region $S_2$ and hence in $S_3$, where $S_2$, $S_3$ are defined in \eqref{S2} and \eqref{S3}, respectively.  Moreover, in the region $S_{2, \varepsilon}$, we have by \eqref{Stirlingratio} that
\begin{align}
\label{BboundS2}
\begin{split}
 \Big|B_j(s,w, z) \Big | \ll  (1+|w|)^{1/2-\Re(w)+\varepsilon}.
\end{split}
\end{align}

  It follows that in the region $S_{3,\varepsilon} \cap \{ (s,w, z)| \Re(w) > 1/2- 1/11 \}$, we have
\begin{align}
\label{Aswboundslarge2Q}
\begin{split}
 |(s-1)B_{j}(s,w,z)| \ll &
\begin{cases}
(1+|s|)^{1+\varepsilon}(1+|w|)^{1/11+\varepsilon}|z|^{\varepsilon}, & \frac 12-\frac 1{11}< \Re(w) < \frac 12, \\
(1+|s|)^{1+\varepsilon}(1+|w|)^{\varepsilon}|z|^{\varepsilon}, & \Re(w) \geq \frac 12.
\end{cases}
\end{split}
\end{align}

   We now proceed as in Section \ref{sec: conclusion} by shifting the line of integration in \eqref{IntegralB} to $\Re(s)=E(\alpha, \beta)+\varepsilon$ to arrive at an asymptotical formula for the first sum on the right-hand side of \eqref{ratioLQ}.  As
\begin{align}
\label{sum12rel}
\begin{split}
 \sumstar_{\varpi}\Lambda_K(\varpi) \frac {L_{\mq}(1/2+\alpha, \overline \chi_{j, \varpi})}{L_{\mq}(1/2+\beta, \overline \chi_{j, \varpi})} \Phi \leg{N(\varpi)}{Q}=\overline{
\sumstar_{\varpi}\Lambda_K(\varpi)\frac {L_{\mq}(1/2+\overline \alpha, \chi_{j, \varpi})}{L_{\mq}(1/2+\overline \beta, \chi_{j, \varpi})}  \Phi \leg{N(\varpi)}{Q}},
\end{split}
\end{align}
  we deduce a similar asymptotical formula for the second sum on the right-hand side of \eqref{ratioLQ}. This then leads to \eqref{Asymptotic for ratios of all charactersQ} and thus completes the proof of Theorem \ref{Theorem for all charactersQ}.

\vspace*{.5cm}

\noindent{\bf Acknowledgments.}  P. G. is supported in part by NSFC grant 12471003 and L. Z. by the FRG Grant PS43707 at the University of New South Wales. \newline

%%\noindent{\bf Data Availability Statement.} This manuscript has no associated data. \newline

%% \noindent{\bf Conflict of Interest Statement.} On behalf of all authors, the corresponding author states that there is no conflict of interest. \newline

\bibliography{biblio}
\bibliographystyle{amsxport}

\end{document}